\documentclass[12pt]{article}

\usepackage{epsfig}
\usepackage{latexsym}
\usepackage{caption}
\usepackage{amsfonts}

\usepackage{amssymb}
\usepackage{mathrsfs}
\usepackage{amsmath}

\usepackage{graphicx}
\usepackage{float}
\usepackage{pict2e}
\usepackage{enumerate}
\usepackage{enumitem}
\usepackage{authblk}

\usepackage{graphics}
\usepackage{tikz}
\usepackage{ulem}
\usepackage{cases}

\usepackage[f]{esvect}

\usepackage{tikz-3dplot}
\usetikzlibrary{hobby}

\usepackage{subcaption}

\usetikzlibrary{shapes.geometric}
\usetikzlibrary{decorations.pathmorphing,calligraphy}

\numberwithin{equation}{section}

\usepackage[pdfauthor={derajan},pdftitle={How to do this},pdfstartview=XYZ,bookmarks=true,
colorlinks=true,linkcolor=blue,urlcolor=blue,citecolor=blue,pdftex,bookmarks=true,linktocpage=true,hyperindex=true]{hyperref}

\usepackage{color}

\definecolor{darkgreen}{RGB}{0,154,23}
\definecolor{darkbrown}{RGB}{183,60,18}

\usepackage[colorinlistoftodos,prependcaption,textsize=scriptsize, textwidth=24mm,]{todonotes}

\usetikzlibrary{patterns}

\makeatother

\begin{document}
	\newtheorem{theorem}{Theorem}
	\newtheorem{observation}[theorem]{Observation}
	\newtheorem{corollary}[theorem]{Corollary}
	\newtheorem{algorithm}[theorem]{Algorithm}
	\newtheorem{definition}[theorem]{Definition}
	\newtheorem{guess}[theorem]{Conjecture}
	
	\newtheorem{problem}[theorem]{Problem}
	\newtheorem{question}[theorem]{Question}
	\newtheorem{lemma}[theorem]{Lemma}
	\newtheorem{proposition}[theorem]{Proposition}
	\newtheorem{fact}[theorem]{Fact}
	
	\newtheorem{claim}{Claim}
	\newtheorem{claimA}{Claim}
	\renewcommand{\theclaimA}{1\alph{claimA}}
	\newtheorem{claimB}{Claim}
	\renewcommand{\theclaimB}{2\alph{claimB}}
	\newtheorem{claimC}{Claim}
	\renewcommand{\theclaimC}{3\alph{claimC}}
	
	\captionsetup[figure]{labelfont={bf},name={Fig.},labelsep=period}
	
	\makeatletter
	\newcommand\figcaption{\def\@captype{figure}\caption}
	\newcommand\tabcaption{\def\@captype{table}\caption}
	\makeatother

	\newtheorem{acknowledgement}[theorem]{Acknowledgement}
	
	\newtheorem{axiom}[theorem]{Axiom}
	\newtheorem{case}[theorem]{Case}
	\newtheorem{conclusion}[theorem]{Conclusion}
	\newtheorem{condition}[theorem]{Condition}
	\newtheorem{conjecture}[theorem]{Conjecture}
	\newtheorem{criterion}[theorem]{Criterion}
	\newtheorem{example}[theorem]{Example}
	\newtheorem{exercise}[theorem]{Exercise}
	\newtheorem{notation}{Notation}
	\newtheorem{solution}[theorem]{Solution}
	\newtheorem{summary}[theorem]{Summary}
	
	\newenvironment{proof}{\noindent {\bf
			Proof.}}{\rule{2.5mm}{2.5mm}\par\medskip}
	\newcommand{\remark}{\medskip\par\noindent {\bf Remark.~~}}
	
	\newcommand{\qed}{\null\nobreak\rule{0.6em}{0.6em}}

	\newcommand{\overbar}[1]{\mkern 1.5mu\overline{\mkern-1.5mu#1\mkern-1.5mu}\mkern 1.5mu}
	
	\let\svthefootnote\thefootnote
	\newcommand\freefootnote[1]{%
		\let\thefootnote\relax%
		\footnotetext{#1}%
		\let\thefootnote\svthefootnote%
	}

	\title{\large{\bf On Two problems of Defective Choosability of  Graphs}}

	\author{Jie Ma \thanks{School of Mathematical Sciences, University of Science and Technology of China, Hefei, Anhui, 230026, China. E-mail: jiema@ustc.edu.cn. Supported in part by the National Key R and D Program of China 2020YFA0713100, National Natural Science Foundation of China grants 12125106, Innovation Program for Quantum Science and Technology 2021ZD0302904, and Anhui Initiative in Quantum Information Technologies grant AHY150200.} \qquad  Rongxing Xu \thanks{School of Mathematical Sciences, University of Science and Technology of China, Hefei, Anhui, 230026, China and School of Mathematical Sciences, Zhejiang Normal University, Jinhua, Zhejiang, 321000, China. E-mail: xurongxing@ustc.edu.cn. Supported by Anhui Initiative in Quantum Information Technologies grant  AHY150200. Supported also by NSFC grant 11871439.} \qquad   Xuding Zhu \thanks{School of Mathematical Sciences, Zhejiang Normal University, Jinhua, Zhejiang, 321000, China. E-mail: xdzhu@zjnu.edu.cn. Grant Number: NSFC 11971438, U20A2068.}}

	\maketitle
	
	\begin{abstract}
		Given positive integers $p \ge k$, and a non-negative integer $d$, we say a graph $G$ is $(k,d,p)$-choosable if for every list assignment $L$ with $|L(v)|\geq  k$ for each $v \in  V(G)$ and $|\bigcup_{v\in V(G)}L(v)| \leq p$, there exists an $L$-coloring of $G$ such that each monochromatic subgraph has maximum degree at most $d$. In particular,  $(k,0,k)$-choosable  means $k$-colorable, $(k,0,+\infty)$-choosable  means $k$-choosable and $(k,d,+\infty)$-choosable means $d$-defective $k$-choosable. This paper proves that there are 1-defective 3-choosable planar graphs that are not 4-choosable, and for any positive integers $\ell \geq k \geq 3$, and non-negative integer $d$, there are $(k,d, \ell)$-choosable graphs that are not $(k,d , \ell+1)$-choosable. These results answer questions asked by Wang and Xu [SIAM J. Discrete Math. 27, 4(2013), 2020-2037], and  Kang [J. Graph Theory 73, 3(2013), 342-353], respectively. Our construction of $(k,d, \ell)$-choosable but not $(k,d , \ell+1)$-choosable graphs generalizes the construction  of Kr\'{a}l' and Sgall in [J. Graph Theory 49, 3(2005), 177-186] for the case $d=0$.
	\end{abstract}
	
	\section{Introduction}
	
	A coloring of a graph $G$ is a mapping $\phi$ which assigns to each vertex $v$ a color. The \emph{defect} of a vertex $v$, denoted by $\lambda_G(v,\phi)$, is the number of neighbors of $v$ which have the same color as $v$.
	A coloring $\phi$ is \emph{$d$-defective}  if $\lambda_G(v,\phi) \leq d$ for each vertex $v \in V(G)$. A $0$-defective coloring is also called a proper coloring.

	Assume $G$ is a graph and $f:V(G)\rightarrow \mathbb{N}^+$ is a mapping.  An \emph{$f$-list assignment} of $G$ is a list assignment $L$ of $G$ which assigns to each vertex $v$ a set $L(v)$ of $f(v)$ colors. Given a list assignment $L$ and nonnegative integer $d$, a {\em $d$-defective $L$-coloring $\phi$}    of $G$ is a $d$-defective coloring $\phi$ of $G$ such that $\phi(v) \in L(v)$  for each vertex $v$. We say $G$ is {\em $d$-defective $f$-choosable} if $G$ has a $d$-defective $L$-coloring for any $f$-list assignment $L$.
	
	Given a mapping $f:V(G)\rightarrow \mathbb{N}^+$ and two nonnegative integers $d, p$, we say   $G$ is \emph{$(f,d,p)$-choosable} if for any $f$-list assignment $L$ with $|\bigcup_{v\in V(G)}L(v)| \leq p$, there exists a $d$-defective $L$-coloring of $G$. In particular,
	$(k,0,k)$-choosable is called \emph{$k$-colorable},   $(k,0,+\infty)$-choosable is called  \emph{$k$-choosable},  $(k,d,k)$-choosable is called \emph{$d$-defective $k$-colorable},  $(k,d,+\infty)$-choosable is called \emph{$d$-defective $k$-choosable}.

	Defective coloring of planar graphs was first studied by Cowen, Cowen and Woodall~\cite{CCW1986}.
	They proved that every outerplanar graph is 2-defective 2-colorable and that every planar graph is 2-defective 3-colorable. These results were strengthened to defective list coloring by  
	Eaton and Hull~\cite{EH1999} and \v{S}krekovski~\cite{Riste1999}  independently. They proved that every planar graph is 2-defective 3-choosable and every outerplanar graph is 2-defective 2-choosable. 
	
	These results motivated some problems on the relation among defective colorability, defective choosability and choosability of planar graphs. 
	It is known that  there are $4$-choosable planar graphs that are not $1$-defective $3$-colorable \cite{WX2013} (hence not $1$-defective 3-choosable).   Wang and Xu \cite{WX2013} asked the following.
	\begin{question}\label{Q3}
		Is every $1$-defective $3$-choosable graph $4$-choosable?
	\end{question}
	
	This paper gives a negative answer to this question in a stronger form as following.
	
	\begin{theorem}\label{main-th1}
		There are  1-defective 3-choosable planar graphs that are   4-choosable.
	\end{theorem}

	\medskip
	
	By definition, for any $d\geq 0$, if $G$ is $(k,d,+\infty)$-choosable, then it is $(k,d,p)$-choosable for any $p \geq k$.   Kr\'{a}l' and Sgall \cite{KS2005} showed that for each $\ell \geq k \geq 3$, there exists a $(k,0,\ell)$-choosable graph which is not $(k,0,\ell+1)$-choosable.  Kang \cite{Kang2013}  asked  the following question.

	\begin{question}
		\label{Q-K}
		Given positive integers $k,d$, does there exist an integer $\ell_{k,d}$ such that every $(k,d, \ell_{k,d})$-choosable graph is $(k,d,+\infty)$-choosable?
	\end{question}
	
	Our second result answers this question in negative  {for   $k \geq 3$, which generalizes the construction of Kr\'{a}l' and Sgall \cite{KS2005} to the cases of $d \geq 0$.}
	
	\begin{theorem}\label{main-th2}
		For any integers $d \geq 0$ and   $\ell \geq k \geq 3$, there exists a $(k,d,\ell)$-choosable graph which is not $(k,d,\ell+1)$-choosable.
	\end{theorem}

	\section{Proof of Theorem \ref{main-th1}}\label{sec1}

	The gadget graph $T$ depicted in Figure \ref{fig-T} was constructed by Gutner \cite{Gutner1996} (see Fig. \ref{fig-T}) and used in the construction of many counterexamples for several topics related to list coloring of planar graphs \cite{Simith2022,Voigt1998,VW1997,XZ2022}.
	
	\begin{figure}[H]
		\centering  
		\begin{tikzpicture}[>=latex,	
			roundnode/.style={circle, draw=black,fill= magenta, minimum size=1.2mm, inner sep=0pt}]
			\node [roundnode] (u) at (0,2){};
			\node [roundnode] (z) at (0,0){};
			\node [roundnode] (u2) at (1.2,1.6/3){};	
			\node [roundnode] (u1) at (-1.2,1.6/3){};
			\node [roundnode] (v2) at (1.2,-1.6/3){};	
			\node [roundnode] (v1) at (-1.2,-1.6/3){};
			\node [roundnode] (v) at (0,-2){};
			\node [roundnode] (x) at (-8/3,0){};
			\node [roundnode] (y) at (8/3,0){};	
			
			\node at (0,2.2){$u$};
			\node at (0,-2.2){$v$};
			
			\node at (-0.85, -0.6){$v_1$};
			\node at (0.85, -0.6){$v_2$};
			\node at (-0.85,0.55){$u_1$};
			\node at (0.85, 0.6){$u_2$};
			
			\node at (0.15, 0.2){$z$};
			
			\node at (-2.7, -0.3){$x$};
			\node at (2.7, -0.3){$y$};
			
			\draw  (u)--(x)--(u1)--(z)--(u2)--(y)--(u)--(u2);
			\draw  (v1)--(u1)--(u)--(z);
			\draw  (u)--(z)--(v);
			\draw  (v)--(x)--(v1)--(z)--(v)--(v2)--(u2);
			\draw  (v1)--(v)--(y)--(v2);
			\draw  (z)--(v2);
		\end{tikzpicture} 
		\caption{The gadget graph $T$ in \cite{Gutner1996}.}
		\label{fig-T}
	\end{figure}
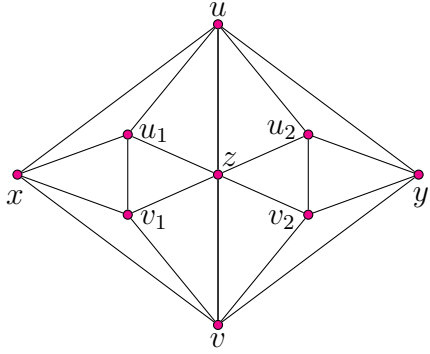
	
	For a positive integer $k$, let $T(k)$ be the graph obtained from the disjoint union of $k$ copies of $T$ by identifying all the copies of top vertex, denoted still by $u$, and identifying all the copies of the bottom vertex, denoted still by $v$. 
	It is known    \cite{Gutner1996,VW1997} that for any $k \geq 16$,  $T(k)$ is not $4$-choosable. 
	
	To prove Theorem \ref{main-th1}, it suffices to show that   if $k \leq 26$, then $T(k)$ is $1$-defective $3$-choosable. 
	
	We first construct some $1$-defective $3$-colorings for   some special list assignments of $T$. We assume $T$ are labeled as in Fig.\ref{fig-T}. 
	
	\begin{lemma}
		\label{lem-added}
		Let $H=T-\{u,v\}$, and $f(w)=2$ for $w \in V(H)\setminus \{x,y,z\}$ and $f(x),f(y), f(z) \ge 1$. If one of $f(x),f(y),f(z) =2$, then $H$ is $1$-defective $f$-choosable.
	\end{lemma}
	\begin{proof}
		Assume $f(y)=f(z)=1$ and $f(x)=2$ and $L$ is an $f$-list assignment of $H$. Let $\phi(y), \phi(z)$ be the unique color in $L(y), L(z)$, respectively. Let $\phi(u_2) \in L(u_2) - \{\phi(y)\}$, $\phi(v_2) \in L(v_2)-\{\phi(z)\}$, $\phi(u_1) \in L(u_1) -\{\phi(z)\}, \phi(v_1) \in L(v_1)-\{\phi(z)\}$ and $\phi(x) \in L(x) - \{\phi(u_1)\}$. It is straightforward to verify that $\phi$ is a 1-defecitve $L$-coloring of $H$.
		
		Assume $f(z)=2$ and $f(x)=f(y)=1$ and $L$ is an $f$-list assignment of $H$. Let $\phi(x), \phi(y)$ be the unique color in $L(x), L(y)$, respectively.
		Let $\phi(u_1) \in L(u_1) - \{\phi(x)\}$, $\phi(z) \in L(z)-\{\phi(u_1)\}$, $\phi(v_1) \in L(v_1)-\{\phi(z)\}$, $\phi(u_2) \in L(u_2) - \{\phi(z)\}$ and $\phi(v_2) \in L(v_1) - \{\phi(y)\}$. Again it is straightforward to verify that $\phi$ is a 1-defective $L$-coloring of $H$.
	\end{proof}

	\begin{corollary}\label{lem-notbad1}
		Let $L$ be a list assignment of $T$ with $L(u)=\{\alpha\}$, $L(v)=\{\beta\}$ and $|L(w)|\geq 3$ for
		$w \in V(T)\setminus \{u,v\}$. If $\alpha = \beta$, or $\alpha \neq \beta$ and
		$\{\alpha, \beta\} \not\subseteq L(x) \cap  L(y) \cap L(z)$, or $\alpha \ne \beta$ and $ L(x) \cap  L(y) \cap L(z)-\{\alpha, \beta\} \ne \emptyset$, then $T$ has a $1$-defective $L$-coloring $\phi$ such that $\lambda_T(u,\phi)=\lambda_T(v,\phi) = 0$.
	\end{corollary}
	\begin{proof}
		Let $L'$ be the list assignment of $H=T-\{u,v\}$ defined as $L'(w) = L(w) - \{\alpha, \beta\}$ if $w \in \{x,y,z\}$,
		and $L'(w)=L(w)-\{\alpha\}$ for $w \in \{u_1,u_2\}$ and $L'(w)=L(w)-\{\beta\}$ for $w \in \{v_1, v_2\}$. If  $\alpha = \beta$ or
		$\{\alpha, \beta\} \not\subseteq L(x) \cap  L(y) \cap L(z)$, then it follows  from
		Lemma \ref{lem-added} that $H$ has a 1-defective $L'$-coloring $\phi$. Extend $\phi$ to $T$ by letting $\phi(u)=\alpha$ and $\phi(v)=\beta$,
		then $\phi$ is a $1$-defective $L$-coloring $\phi$ such that $\lambda_T(u,\phi)=\lambda_T(v,\phi)= 0$.
		
		If $\alpha \ne \beta$ and $ L(x) \cap  L(y) \cap L(z)-\{\alpha, \beta\} \ne \emptyset$, say $c \in L(x) \cap  L(y) \cap L(z)-\{\alpha, \beta\}$, then let $\phi(x)=\phi(y)=\phi(z)=c$, $\phi(w) \in L(w) - \{\alpha, c\}$ for $w \in \{u_1,u_2\}$ and $\phi(w) \in L(w) - \{\beta, c\}$ for $w \in \{v_1,v_2\}$.
		It is straightforward to verify that  $\phi$ is a $1$-defective $L$-coloring $\phi$ such that $\lambda_T(u,\phi)=\lambda_T(v,\phi)= 0$.
	\end{proof}

	\begin{lemma}
		\label{lem-uv}
		Let $L$ be a list assignment of $T$ with $L(u)=\{\alpha\}$, $L(v)=\{\beta\}$ and $|L(w)|\geq 3$ for	$w \in V(T)\setminus \{u,v\}$. Then $T$ has a $1$-defective $L$-coloring $\phi$ such that $\lambda_T(u,\phi) = 0$, and a $1$-defective $L$-coloring $\phi$ such that $\lambda_T(v,\phi) = 0$.
	\end{lemma}
	\begin{proof}
		By Corollary \ref{lem-notbad1}, it suffices to consider the case that $\alpha \neq \beta$ and
		$\{\alpha,\beta\}\subseteq L(x) \cap L(y) \cap L(z)$. Let    $\phi(u)=\alpha, \phi(v)=\beta$, $\phi(x), \phi(z)$ be the unique color in $L(x)-\{\alpha, \beta\}, L(z) - \{\alpha, \beta\}$, respectively. Let $\phi(u_1) \in L(u_1) - \{\alpha, \phi(x)\}, \phi(v_1) \in L(v_1)-\{\beta, \phi(z)\}$,
		$\phi(u_2) \in L(u_2)-\{\alpha, \phi(z)\}, \phi(v_2) \in L(v_2)-\{\beta, \phi(z)\}$.
		If $\phi(u_2) \ne \phi(v_2)$, then let  $\phi(y)\in L(y) - \{\phi(u_2), \phi(v_2)\}$. Otherwise, let
		$\phi(y)$ be any color in $\{\alpha, \beta\}$. It is easy to verify that $\phi$ is a $1$-defective coloring of $T$. In most cases, $\lambda_T(u,\phi)=\lambda_T(v,\phi) = 0$, except that in the last case, if $\phi(y)=\alpha$, then $\lambda_T(v,\phi) = 0$, if $\phi(y)=\beta$, then $\lambda_T(u,\phi)= 0$.
	\end{proof} 
	
	Now we are ready to prove Theorem \ref{main-th1}.
	
	\noindent
	\textbf{Proof of Theorem \ref{main-th1}.}  
	Let $L$ be any $3$-list assignment of $T(k)$, we shall give a $1$-defective $3$-coloring $\phi$ of $T(k)$. Without loss of generality, we assume that
	$|L(w)|=3$ for each vertex $w$.
	If $L(u) \cap L(v) \ne \emptyset$, say $\alpha \in L(u) \cap L(v)$, then let $\phi(u)=\phi(v) = \alpha$. By Corollary \ref{lem-notbad1},
	$\phi$ can be extended to a 1-defective coloring of $T(k)$ so that $\lambda_T(u,\phi)=\lambda_T(v,\phi) = 0$.
	
	Assume $L(u) \cap L(v) = \emptyset$. For each pair of colors $(\alpha, \beta)$ with $\alpha \in L(u), \beta\in L(v)$,
	a copy of $T$ is called \emph{$(\alpha, \beta)$-tight} if $\{\alpha, \beta \} = L(x) \cap L(y) \cap L(z)$ (here $x,y,z$ refer to the corresponding vertices in that copy of $T$). Note that if a copy of $T$ is $(\alpha, \beta)$-tight, then it is not $(\alpha', \beta')$-tight for any $(\alpha', \beta') \ne (\alpha, \beta)$.
	So by pigeonhole principle, there is a pair $(\alpha, \beta)$ with $\alpha \in L(u), \beta\in L(v)$ such that there are at most two copies of $T$, say $T^1$ and $T^2$, that are
	$(\alpha, \beta)$-tight. Let $\phi(u)=\alpha, \phi(v)=\beta$. By Corollary \ref{lem-notbad1}, for each other copy $T'$ of $T$,
	$\phi$ can be extended to a 1-defective $L$-coloring of $T'$ so that $\lambda_{T'}(u,\phi)=\lambda_{T'}(v,\phi) = 0$. By Lemma \ref{lem-uv}, $\phi$ can be extended to a 1-defective $L$-coloring of $T^1$ so that
	$\lambda_{T^1}(u,\phi)= 0$, and $\phi$ can be extended to a 1-defective $L$-coloring of $T^2$ so that
	$\lambda_{T^2}(v,\phi)= 0$. Thus $\phi$ can be extended to a 1-defective $L$-coloring of $T(k)$.  \qed
	
	\section{Proof of Theorem \ref{main-th2}}\label{sec2}
	Let $\ell\geq k \geq 3$ and $d \geq 0$ be any fixed integers.
	In this section we prove Theorem \ref{main-th2} by constructing graphs that are $(k,d,\ell)$-choosable but not $(k,d, \ell+1)$-choosable.

	Lov\'{a}sz \cite{Lovasz1966} proved that every graph $G$ with $\Delta(G)+1 \leq k(d+1)$ is $d$-defective $k$-colorable. This was generalized to list-version by  Hendrey and Wood  (see Corollary 3.2 in \cite{HW2019}).
	\begin{lemma}[\cite{HW2019}]\label{lem-max-degree}
		Every graph $G$ with $\Delta(G)+1 \leq k(d+1)$ is $d$-defective $k$-choosable.
	\end{lemma}
	
	Lemma~\ref{lem-max-degree} implies that every complete graph with $k(d+1)$ vertices is $d$-defective $k$-choosable. In this paper, we need the following slightly stronger statement.
	\begin{lemma}\label{lem-complete-extend}
		Suppose $G$ is a complete graph of order $k(d+1)$, and $v$ is any vertex in $G$. Assume $L$ is a list assignment with $|L(u)| \geq k$ for any $u \in V(G)\setminus \{v\}$ and $|L(v)|=1$. Then $G$ has a $d$-defective $L$-coloring.
	\end{lemma}
	\begin{proof}
		Assume $L(v)=\{c\}$. Let $V_c=\{u \in V(G): c \in L(u)\}$. If $|V_c| \geq d$, then we choose arbitrary $d$-subset $D \subset V_c$, and color all the vertices in $D$ and $v$ with color $c$, also we delete the color $c$ from $V_c-D$. Note that $G-v-D$ is a complete graph of order $(k-1)(d+1)$, and each vertex has at least $(k-1)$ colors available. By Lemma \ref{lem-max-degree}, we can color $G-v-D$ without using $c$ and with defect $d$ by colors from the lists. So assume that $|V_c| \leq d-1$. Then we color all the vertices in $V_c$   with color $c$. Note that $G'=G-V_c$ still satisfies that $\Delta(G')+ 1 \leq k(d+1)$ and each vertex has $k$ colors available. Then we can extend the coloring to whole $G$ by Lemma \ref{lem-max-degree}.
	\end{proof}

	The following lemma follows from Lemma 3.3 in \cite{HW2019}.
	
	\begin{lemma}[\cite{HW2019}] \label{lem-extend}
		Let $L$ be a $k$-list assignment of a graph $G$. Let $A,B$ be a partition of $V(G)$, where $G[A]$ is $d$-defective  $L$-colorable. If  for every vertex $v \in B$,
		$$(d+1)|N_G(v) \cap A| +\deg_{B}(v)+1 \leq (d+1)k,$$
		then $G$ is $d$-defective $L$-colorable.
	\end{lemma}
	
	Given a graph $G$ and an non-negative integer $d$, we denote by $G  \ast d$  the graph obtained from the disjoint union of $G$ and $|V(G)|$   copies of the complete graph $K_{d}$, denoted as $\{B_v: v \in V(G)\}$,  by identifying  $v$ with one vertex of $B_v$.
	Fig. \ref{fig-susp} shows the graph $C_9 \ast 4$.
	
	\begin{figure}[H]
		\centering
		\begin{tikzpicture}[>=latex,	
			roundnode/.style={circle, draw=black,fill= magenta, minimum size=1.2mm, inner sep=0pt},
			bluenode/.style={circle, draw=black,fill= green, minimum size=1.2mm, inner sep=0pt}]
			\foreach \t in {1,2,...,9}{
				\node[roundnode]  (A\t) at (40*\t-40:1.2) {};}
			
			\foreach \t in {1,2,...,9}{
				\node[bluenode]  (B\t) at (40*\t-49:1.46) {};}
			
			\foreach \t in {1,2,...,9}{
				\node[bluenode]  (C\t) at (40*\t-31:1.46) {};}
			
			\foreach \t in {1,2,...,9}{
				\node[bluenode]  (D\t) at (40*\t-40:1.7) {};}

			\foreach \t in {1,2,...,9}{
				\draw[fill=cyan!30] (A\t)--(B\t)--(C\t)--(D\t)--(A\t);
				\draw (A\t)--(C\t);
				\draw (B\t)--(D\t);}
			
			\draw (A1)--(A2)--(A3)--(A4)--(A5)--(A6)--(A7)--(A8)--(A9)--(A1);
			
			\foreach \t in {1,2,...,9}{
				\node at (40*\t-40:0.9) {$v_{\t}$};}
		\end{tikzpicture}
		\caption{$C_9 \ast 4$}
		\label{fig-susp}
	\end{figure}
	
	\begin{lemma}\label{lem-product-no}
		$H=C_{2k+1} \ast (2d+2)$ is not $d$-defective $2$-colorable.
	\end{lemma}
	\begin{proof}
		Let $v_1,v_2,\ldots,v_{2k+1}$ be the vertices of $C_{2k+1}$ in order, and $B_1,B_2,\ldots, B_{2k+1}$ be the corresponding vertex-set of the $2k+1$ copies of $K_{2d+2}$. Assume there is a $d$-defective $2$-coloring $\phi$ of $H$. Since $B_i$ is a clique of order $2(d+1)$,  each of the two colors is used $d+1$ times by $\phi$ in $B_i$. Therefore, $\phi(v_i) \neq \phi(v_j)$ whenever $v_iv_j\in E(C_{2k+1})$. This is  a contradiction, as $C_{2k+1}$ is not $2$-colorable.
	\end{proof}
	
	\begin{lemma}\label{lem-product-yes}
		Let $H=C_{2k+1} \ast (2d+2)$, with   $v_1,v_2,\ldots,v_{2k+1}$ be the   vertices of $C_{2k+1}$ in this cyclic order. Let $B_1,B_2,\ldots, B_{2k+1}$ be the corresponding vertex-set of the $2k+1$ copies of $K_{2d+2}$. Assume  $L$ is a list assignment of $H$ with $|L(u)| \geq 2$ for $u \in \bigcup^{2k}_{i=1}B_i$ and $|L(u)| \geq 3$ for $u \in B_{2k+1}$. Then for any color $c$,   $H$ has a  $d$-defective $L$-coloring $\phi$ such  that at most one vertex in $B_{2k+1}$ is colored with $c$ in $\phi$.
	\end{lemma}
	\begin{proof}
		As  $|L(v_{2k+1})| \geq 3$, there is a proper $L$-coloring $\phi$ of $C_{2k+1}$.
		We then extend $\phi$ from each $v_i$ to $B_i$ by Lemma~\ref{lem-complete-extend} so that   color $c$ is not used on other vertices of $B_{2k+1}$. This is possible since each vertex in $B_{2k+1}$ still has at least two colors available.
	\end{proof}
	%

	\begin{lemma}\label{lem-Htdk}
		Assume $k \ge 3$, $t \ge 2$, $d \ge 0$ and $\ell=k-2+t$. There exists a  graph $H(t,d,k)=(V,E)$ with a precolored independent set $T=\{u_1,u_2,\ldots,u_t\}$ for which   the following hold:
		\begin{itemize}  
			\item  Assume the precoloring $\phi$ of $T$ uses   $t$ distinct colors {in $[\ell+1]$}. Then   there is a $k$-list assignment $L$ of $H(t,d,k)$ with 
			$ L(v) \subseteq [\ell+1]$ for each vertex $v$ such that $\phi$ cannot be extended to  a $d$-defective coloring $\psi$ of $H(t,d,k)$ with $\lambda_{H(t,d,k)}(u_i,\psi)=0$ for each $u_i \in T$. On the other hand, if $d \ge 1$, then
			for any $k$-list assignment $L$ of $H(t,d,k)-T$,   $\phi$ can be extended to a $d$-defective $L$-coloring $\psi$ of $H(t,d,k)$ such that $\lambda_{H(t,d,k)}(u_i,\psi)=0$ for $i=1,2,\ldots, t-1$ and $\lambda_{H(t,d,k)}(u_t,\psi) \leq 1$. 
			\item  Assume the  precoloring $\phi$ of $T$ uses with at most $t-1$ colors. Then  for any $k$-list assignment $L$ of $H(t,d,k)-T$,   $\phi$   can be extended to a $d$-defective $L$-coloring $\psi$ of   $H(t,d,k)$ such that $\lambda_{H(t,d,k)}(u_i,\psi)=0$ for each $u_i \in T$.
		\end{itemize}
	\end{lemma}

	\begin{proof}
		We consider three cases according to the values of $\ell$ and $k$.
		
		\medskip
		\noindent
		\textbf{Case 1.}   $t \leq k-1$ (hence $3 \leq k \leq  \ell \leq 2k-3$).
		
		Let  $H(t,d,k)$ be the join of a complete graph $K_{(d+1)(k-t)+1}$ and an independent set $T = \{u_1,u_2,\ldots, u_t\}$. Assume $\phi$ is a precoloring of $T$.
		
		Assume $\phi$ uses $t$ distinct colors  {in $[\ell+1]$}.  Let  $S$ be $k$ colors in $[\ell+1]$ that contains $\phi(T)$, and let 
		$L(v)=S$ for all $v \in V(H(t,d,k)-T$.
		
		We then show that $\phi$ cannot be extended to a $d$-defective coloring $\psi$ of $H(t,d,k)$ with $\lambda_{H(t,d,k)}(u_i,\psi)=0$. 
		Assume to the contrary that $\phi$ is extended to a $d$-defective coloring $\psi$ of $H(t,d,k)$ with $\lambda_{H(t,d,k)}(u_i,\psi)=0$. Then vertices in the complete graph $K_{(d+1)(k-t)+1}$ are   colored by $k-t$ colors. So one color class contains at least $d+2$ vertices, which is a contradiction. On the other hand, if $d \ge 1$, then for any  
		$k$-list assignment $L$ of  $H(t,d,k)-T$, let $w$ be an arbitrary vertex in $V(H(t,d,k)-T$, 
		$L'(v)=L(v) - \phi(T)$ for $v\in V(H(t,d,k)-T - \{w\}$, and $L'(w) = L(w) - \phi(T-\{u_t\})$. We have $|L'(v)| \geq k-t$ for any $v \in V(K_{(d+1)(k-t)+1})-\{w\}$, and $|L'(w)| \geq k-t+1$. By Lemma \ref{lem-max-degree}, $K_{(d+1)(k-t)+1}$ has a $d$-defective $L'$-coloring, which is an extension of    $\phi$ to a  $d$-defective $L$-coloring of $H(t,d,k)$ with $\lambda_{H(t,d,k)}(u_i,\psi)=0$ for each $u_i \in T-\{u_t\}$, and $\lambda_{H(t,d,k)}(u_t,\psi)\le 1$.

		If $\phi$ use at most $t-1$  colors, then for any  
		$k$-list assignment $L$ of  $V(H(t,d,k))-T$, let $L'(v)=L(v) - \phi(T)$ for $v\in V(H(t,d,k))-T$. We have $|L'(v)| \geq k-t+1$ for any $v \in V(K_{(d+1)(k-t)+1})$. By Lemma \ref{lem-max-degree}, $K_{(d+1)(k-t)+1}$ has a $d$-defective $L'$-coloring, which is an extension of    $\phi$ to a  $d$-defective $L$-coloring of $H(t,d,k)$ with $\lambda_{H(t,d,k)}(u_i,\psi)=0$ for each   $u_i \in T$.

		\medskip
		\noindent
		\textbf{Case 2.} $t \geq k = 3$ (hence $\ell \geq 2k-2=4$).
		
		Let $T=\{u_1,u_2,\ldots, u_t\}$ be an independent set.
		Let $s \in\{ {t \choose 2}, {t \choose 2}+1\}$ be an odd integer, and let $\mathcal{X}=\{X_a: 1\le a \le s\}$  be a set disjoint copies of   $K_{d+1}$. Let $\pi: [s] \to \{(i,j): 1 \le i < j \le t\}$ be a surjective map, and for each $a \in [s]$ with $\pi(a)=(i,j)$, connect each vertex of $X_a$ to $u_i$ and $u_j$. 
		Add the graph $H = C_{s} \ast 2(d+1)$, with $\{B_a: 1 \le a \le s\}$   be the corresponding vertex-set of the $s$  copies of $K_{2d+2}$. For each $a \in [s]$, choose one vertex   $x_a \in X_a$  and connect $x_a$ to all the vertices in $B_a$. 
		This completes the construction of   $H(t,d,3)$.  Fig.~\ref{fig-Ht3} is an example with $t=4$ and $d=3$.
		
		\begin{figure}[htbp]
			\centering
			\centering
			\begin{tikzpicture}[>=latex,	
				Rnode/.style={circle, draw=black,fill= magenta, minimum size=1.2mm, inner sep=0pt},
				Ynode/.style={circle, draw=black,fill= yellow, minimum size=1.2mm, inner sep=0pt},
				Gnode/.style={circle, draw=black,fill= green, minimum size=1.2mm, inner sep=0pt},
				Bnode/.style={circle, draw=black,fill= cyan, minimum size=1.5mm, inner sep=0pt}]
				
				\foreach \t in {1,2,...,7}{
					\draw[fill=cyan, opacity=0.3] (1.5*\t-1.5,0) circle[radius =0.5cm];
					\draw[fill=cyan, opacity=0.3] (1.5*\t-1.5,1.5) circle[radius =0.4cm];}
				

				\node [Rnode] (A3) at ($(0,0)+(90:0.4)$){};
				\node [Rnode] (B3) at ($(1.5,0)+(90:0.4)$){};
				\node [Rnode] (C3) at ($(3,0)+(90:0.4)$){};
				\node [Rnode] (D3) at ($(4.5,0)+(90:0.4)$){};
				\node [Rnode] (E3) at ($(6,0)+(90:0.4)$){};
				\node [Rnode] (F3) at ($(7.5,0)+(90:0.4)$){};
				\node [Rnode] (G3) at ($(9,0)+(90:0.4)$){};

				\foreach \t in {1,2,4,5,6,7,8}{
					\node[Ynode]  (A\t) at ($(0,0)+(45*\t-45:0.4)$) {};}
				\foreach \t in {1,2,4,5,6,7,8}{
					\node[Ynode]  (B\t) at ($(1.5,0)+(45*\t-45:0.4)$) {};}	
				\foreach \t in {1,2,4,5,6,7,8}{
					\node[Ynode]  (C\t) at ($(3,0)+(45*\t-45:0.4)$) {};}		
				\foreach \t in {1,2,4,5,6,7,8}{
					\node[Ynode]  (D\t) at ($(4.5,0)+(45*\t-45:0.4)$) {};}
				\foreach \t in {1,2,4,5,6,7,8}{
					\node[Ynode]  (E\t) at ($(6,0)+(45*\t-45:0.4)$) {};}
				\foreach \t in {1,2,4,5,6,7,8}{
					\node[Ynode]  (F\t) at ($(7.5,0)+(45*\t-45:0.4)$) {};}	
				\foreach \t in {1,2,4,5,6,7,8}{
					\node[Ynode]  (G\t) at ($(9,0)+(45*\t-45:0.4)$) {};}		
				
				\foreach \t in {1,2,...,8}{
					\foreach \r in {1,2,...,8}{	
						\draw (A\t)--(A\r);
						\draw (B\t)--(B\r);
						\draw (C\t)--(C\r);
						\draw (D\t)--(D\r);
						\draw (E\t)--(E\r);
						\draw (F\t)--(F\r);
						\draw (G\t)--(G\r);}}	
				
				\node [Gnode] (H4) at ($(0,1.5)+(270:0.3)$){};
				\node [Gnode] (I4) at ($(1.5,1.5)+(270:0.3)$){};
				\node [Gnode] (J4) at ($(3,1.5)+(270:0.3)$){};
				\node [Gnode] (K4) at ($(4.5,1.5)+(270:0.3)$){};
				\node [Gnode] (L4) at ($(6,1.5)+(270:0.3)$){};
				\node [Gnode] (M4) at ($(7.5,1.5)+(270:0.3)$){};
				\node [Gnode] (N4) at ($(9,1.5)+(270:0.3)$){};
				
				\foreach \t in {1,2,3}{
					\node[Ynode]  (H\t) at ($(0,1.5)+(90*\t-90:0.3)$) {};
					\node[Ynode]  (I\t) at ($(1.5,1.5)+(90*\t-90:0.3)$) {};
					\node[Ynode]  (J\t) at ($(3,1.5)+(90*\t-90:0.3)$) {};
					\node[Ynode]  (K\t) at ($(4.5,1.5)+(90*\t-90:0.3)$) {};
					\node[Ynode]  (L\t) at ($(6,1.5)+(90*\t-90:0.3)$) {};
					\node[Ynode]  (M\t) at ($(7.5,1.5)+(90*\t-90:0.3)$) {};
					\node[Ynode]  (N\t) at ($(9,1.5)+(90*\t-90:0.3)$) {};}
				
				\foreach \t in {1,2,3,4}{
					\foreach \r in {1,2,3,4}{
						\draw (H\t)--(H\r);	
						\draw (I\t)--(I\r);	
						\draw (J\t)--(J\r);	
						\draw (K\t)--(K\r);	
						\draw (L\t)--(L\r);
						\draw (M\t)--(M\r);
						\draw (N\t)--(N\r);}}

				\draw [line width =1.2pt] ($(0,0)+(135:0.6)$)--(H4);
				\draw [line width =1.2pt] ($(0,0)+(90:0.55)$)--(H4);
				\draw [line width =1.2pt] ($(0,0)+(45:0.6)$)--(H4);
				
				\draw [line width =1.2pt] ($(1.5,0)+(135:0.6)$)--(I4);
				\draw [line width =1.2pt] ($(1.5,0)+(90:0.55)$)--(I4);
				\draw [line width =1.2pt] ($(1.5,0)+(45:0.6)$)--(I4);
				
				\draw [line width =1.2pt] ($(3,0)+(135:0.6)$)--(J4);
				\draw [line width =1.2pt] ($(3,0)+(90:0.55)$)--(J4);
				\draw [line width =1.2pt] ($(3,0)+(45:0.6)$)--(J4);
				
				\draw [line width =1.2pt] ($(4.5,0)+(135:0.6)$)--(K4);
				\draw [line width =1.2pt] ($(4.5,0)+(90:0.55)$)--(K4);
				\draw [line width =1.2pt] ($(4.5,0)+(45:0.6)$)--(K4);
				
				\draw [line width =1.2pt] ($(6,0)+(135:0.6)$)--(L4);
				\draw [line width =1.2pt] ($(6,0)+(90:0.55)$)--(L4);
				\draw [line width =1.2pt] ($(6,0)+(45:0.6)$)--(L4);
				
				\draw [line width =1.2pt] ($(7.5,0)+(135:0.6)$)--(M4);
				\draw [line width =1.2pt] ($(7.5,0)+(90:0.55)$)--(M4);
				\draw [line width =1.2pt] ($(7.5,0)+(45:0.6)$)--(M4);
				
				\draw [line width =1.2pt] ($(9,0)+(135:0.6)$)--(N4);
				\draw [line width =1.2pt] ($(9,0)+(90:0.55)$)--(N4);
				\draw [line width =1.2pt] ($(9,0)+(45:0.6)$)--(N4);

				\node [Bnode] (T1) at (1.5,4) {};
				\node [Bnode] (T2) at (3.5,4) {};
				\node [Bnode] (T3) at (5.5,4) {};
				\node [Bnode] (T4) at (7.5,4) {};
				
				\draw [magenta,thick] (A3)--(B3)--(C3)--(D3)--(E3)--(F3)--(G3) ..controls (10,0.8) and (-1,0.8).. (A3);

				\draw [line width =1.5pt] ($(0,1.5)+(90:0.43)$)--(T1);
				\draw [line width =1.5pt] ($(0,1.5)+(50:0.43)$)--(T2);
				\draw [line width =1.5pt] ($(1.5,1.5)+(90:0.43)$)--(T1);
				\draw [line width =1.5pt] ($(1.5,1.5)+(50:0.43)$)--(T3);
				\draw [line width =1.5pt] ($(3,1.5)+(95:0.43)$)--(T1);
				\draw [line width =1.5pt] ($(3,1.5)+(75:0.43)$)--(T4);
				\draw [line width =1.5pt] ($(4.5,1.5)+(90:0.43)$)--(T2);
				\draw [line width =1.5pt] ($(4.5,1.5)+(75:0.43)$)--(T3);
				\draw [line width =1.5pt] ($(6,1.5)+(95:0.43)$)--(T2);
				\draw [line width =1.5pt] ($(6,1.5)+(75:0.43)$)--(T4);
				\draw [line width =1.5pt] ($(7.5,1.5)+(100:0.43)$)--(T3);
				\draw [line width =1.5pt] ($(7.5,1.5)+(90:0.43)$)--(T4);
				\draw [line width =1.5pt] ($(9,1.5)+(110:0.43)$)--(T3);
				\draw [line width =1.5pt] ($(9,1.5)+(90:0.43)$)--(T4);
				
				\node at (1.5,4.25) {$u_1$};
				\node at (3.5,4.25) {$u_2$};
				\node at (5.5,4.25) {$u_3$};
				\node at (7.5,4.25) {$u_4$};
				
				\node at (10,0) {$H$};
				
				\foreach \t in {1,2,...,7}{
					\node at (1.5*\t-1.5,-0.8) {$B_{\t}$};}
				\foreach \t in {1,2,...,7}{
					\node at (1.5*\t-1,1) {\small $X_{\t}$};}
			\end{tikzpicture}
			
			\caption{$H(4,3,3)$; The green vertices are $x_i$s and red vertices are $v_i$s in the copy of $C_7\ast 8$. We used the pair $(u_3,u_4)$ twice as $\binom{4}{2}$ is even.}
			\label{fig-Ht3}
		\end{figure}
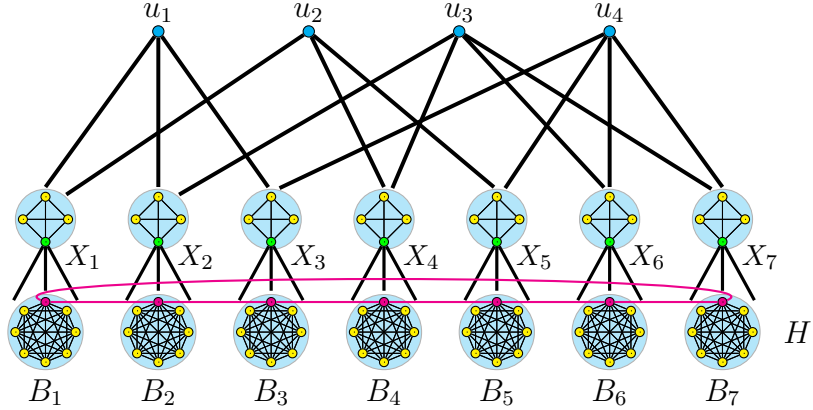
		
		Assume $\phi$ is a precoloring of $T$. 
		
		Assume $\phi$ uses $t$ distinct colors in $[\ell+1]$. 
		Let $L$ be a 3-list assignment of $H(t,d,3)-T$  defined as follows:  
		For $x \in X_a$ with $a \in [s]$ and $\pi(a)=(i,j)$, let $L(x)=\{\phi(u_i), \phi(u_j),\beta\}$
		and   {$L(y)=\{\phi(u_1),\phi(u_2),\beta\}$ for $y \in V(H)$,} where $\beta \in [\ell+1]\setminus \phi(T)$ is an fixed color.  We shall show that $\phi$ cannot be extended to a $d$-defective coloring $\psi$ of $H(t,d,3)$ with $\lambda_{H(t,d,3)}(u_i,\psi)=0$ for $i=1,2,\ldots, t$. Otherwise, assume
		$\phi$ is extended to a $d$-defective coloring $\psi$ of $H(t,d,3)$ with $\lambda_{H(t,d,3)}(u_i,\psi)=0$. Then vertices in $X_a$ are all colored by $\beta$ for each $a \in [s]$, and hence vertices in $H$ are colored with  {$\phi(u_1)$ and $\phi(u_2)$}. This is in contrary to Lemma \ref{lem-product-no}.

		On the other hand, for any $k$-list assignment $L$ of $H(t,d,3)-T$, let 
		$L'$ be the list assignment of $H$ obtained from $L$ by deleting all the colors used in $\phi$ on  $\{x_1,x_2,\ldots,x_{s-1}\}$.  It is clear that
		\[
		|L'(w)| \geq
		\begin{cases} 2, &\text{if $w \in \bigcup^{s-1}_{i=1}B_i$}, \cr
			3, & \text{if $w \in B_s$}.  
		\end{cases}
		\]
		Therefore, by Lemma \ref{lem-product-yes}, we can extend the $\phi$ to a $d$-defective coloring of $H$, such that at most one vertex in $B_s$ uses $\phi(x_s)$, which yields a $d$-defective $L$-coloring as desired.

		Assume  $\phi$ uses at 
		most $t-1$ distinct colors. Then  at least two vertices $u_i,u_j$ in $T$ received the same color.    Without loss of generality, assume $\pi(s)=(i,j)$.  
		Then we extend $\phi$   to $V(H(t,d,3)) \setminus V(H)$ by assigning  $\phi(w) \in L(w) \setminus \phi(N_{H(t,d,3)}(w)) \cap T)$ for each vertex $w \in \bigcup^{s}_{i=1}V(X_i)$
		such that $w$ has defect at most $d$  if $w\ne x_s$ and at most $d-1$ if $w=x_s$, this is possible as each $X_i$ has exactly $d+1$ vertices and $x_s$ has two colors available.

		Finally, we need to extend $\phi$   to $H$. Let $L'$ be the list assignment of $H$ defined as
		\[
		L'(v)= \begin{cases} L(v)- \{\phi(x_i): i=1,2,\ldots, s\}, &\text{if $v \in V(H) \setminus B_s$}, \cr 
			L(v) - \{\phi(x_i): i=1,2, \ldots, s-1\}, &\text{ if $v \in B_s$}.
		\end{cases}
		\]
		
		Then $|L'(u)| \geq 2$ for $u \in V(H)\setminus B_s$, and $|L'(u)| \geq 3$ if $u \in B_s$. Since $d\geq 1$, then by Lemma~\ref{lem-product-yes}, we can extend $\phi$ to a $d$-defective $L'$-coloring of $H$ such that at most one vertex in $B_s$ is colored with $\phi(x_s)$, which yields a desired coloring.

		\medskip
		\noindent
		\textbf{Case 3.} $t \geq k > 3$ (hence $\ell \geq 2k-2 \geq 6$).
		
		Let $T=\{u_1,u_2,\ldots, u_t\}$ be an independent set. 
		Let $s \in \{{t \choose k-2}, {t \choose k-2}+1\}$ be an odd integer. Denote by ${[t] \choose k-2}$ the set of $(k-2)$-subsets of $[t]=\{1,2,\ldots, t\}$. 
		Let  $H$ be a  copy of $C_{s}\ast (2d+2)$, $B_1,B_2,\ldots, B_{s}$ be the corresponding vertex-set of the $s$  copies of $K_{2d+2}$. Let $\pi: [s] \to {[t] \choose k-2}$  be a surjective map.    
		For $a \in [s]$, connect each vertex in $B_a$ to $\{u_i: i \in \pi(a)\}$.  
		This completes the construction of $H(t,d,k)$.

		Suppose $\phi$ is a precoloring of $T$ with $t$ distinct colors in $[\ell+1]$. For each $B_i$, let $T_i \subset T$ be the neighbors of vertices of $B_i$. Let $L$ be a $k$-list assignment for vertices in $H$, such that for any $w \in B_i$, $L(w)=\bigcup_{v\in T_i}\{\phi(v)\}\cup \{\beta_1,\beta_2\}$, where $\beta_1,\beta_2 \in [\ell+1]\setminus \phi(T)$ (note that $\ell+1 - t \ge 2$). If $\phi$ can be extended to a $d$-defective coloring $\psi$ of $H(t,d,k)$ with $\lambda(u_i,\psi)=0$, then vertices in $H$ can be only colored with two colors $\{\beta_1,\beta_2\}$. But this is not possible by Lemma \ref{lem-product-no}.
		
		On the other hand, if $d \geq 1$, for any $k$-list assignment $L$ of $H(t,d,3)-T$, let 
		$L'$ be the list assignment of $H$ obtained from $L$ as follows (without loss of generality, we assume that $t \in \pi(s)$): 
		\[
		L'(w) =
		\begin{cases} L(w)\setminus \phi(T), &\text{if $w \in \bigcup^{s-1}_{i=1}B_i$}, \cr
			L(w)\setminus \phi(T-u_t), & \text{if $w \in B_s$}.  
		\end{cases}
		\]
		It is clear that 
		$|L'(w)| \geq 2$  if $w \in \bigcup^{s-1}_{i=1}B_i$ and  $|L'(w)| \geq 3$ if $w \in B_s$.   
		Therefore, by Lemma \ref{lem-product-yes}, we can extend the $\phi$ to a $d$-defective coloring of $H$, such that at most one vertex in $B_s$ uses $\phi(x_s)$, which yields a $d$-defective $L$-coloring as desired.

		Assume that $\phi$ is a precoloring of $T$ with $t-1$ distinct colors, and $L$ is a $k$-list assignment of $H$. It follows that at least two vertices in $T$ received same color, say $u_1$ and $u_2$. Without loss of generality, assume that $\{1,2\} \subset \pi(s)$. Let $L'$ be the list assignment of $H$ obtained from $L$ by deleting all the colors used in $\phi$ on $T$. It is clear that $|L(w)|\geq 2$ and moreover, $|L(w)| \geq 3$ if $w \in B_s$. So we can extend the precoloring to $H$ by Lemma~\ref{lem-product-yes} as desired.
	\end{proof}
	
	\medskip
	
	Now we are ready to prove Theorem \ref{main-th2}.

	\noindent
	\textbf{Proof of Theorem \ref{main-th2}.} 
	Assume $\ell \geq k \geq 3$ and $d \geq 0$.
	We shall construct a graph which is $(k,d,\ell)$-choosable but not $(k,d,\ell+1)$-choosable.
	
	Let $t=\ell +2-k \geq 2$. Let $G_1=K_{k(d+1)}$, and $z \in V(G_1)$. Take $k$   disjoint copies of $K_{(k-1)(d+1)}$, which are denoted as $C_1,C_2,\ldots C_k$.  Connect all the vertices of $C_1 \cup \ldots \cup C_k$ to the vertex $z$. 
	The resulting graph is $G_2$. 
	
	Let $q= \binom{(k-1)(d+1)}{k-1}$. For $i=1,2,\ldots, k$, let $X_{i,1}, X_{i,2}, \ldots, X_{i,q}$ be the family of all the  $(k-1)$-subsets   of $V(C_i)$.  For $1 \le i \le k, 1 \le j \le q$,   take $t$ copies of disjoint complete graphs $K_{d+1}$, which are denoted by 
	$D_{i,j,1}, D_{i,j,2}, \ldots, D_{i,j,t}$.
	For $s=1,2,\ldots, t$, connect all the   vertices in $D_{i,j,s}$ to all vertices in $X_{i,j}$.   
	The resulting graph is $G_3$. 
	
	For $1 \le i \le k, 1 \le j \le q$,  take a $t$-set $T_{i,j}=\{u_{i,j,1},u_{i,j,2},\ldots,u_{i,j,t}\}$ with $u_{i,j,s} \in V(D_{i,j,s})$, and build a copy $H_{i,j}(t,d,k)$ of $H(t,d,k)$ described in Lemma \ref{lem-Htdk} with $T=T_{i,j}$. This completes the construction of the graph $G$. See Fig.~\ref{fig-main} for illustration.

	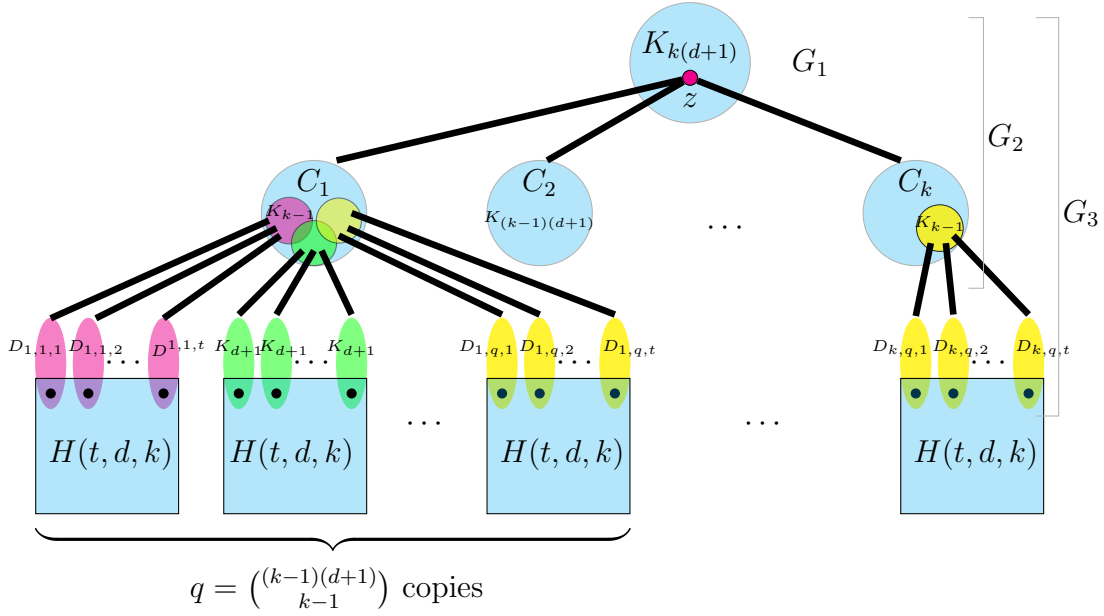
\begin{figure}[h]
		\centering
		\centering
		\begin{tikzpicture}[>=latex,	
			Rnode/.style={circle, draw=black,fill= magenta, minimum size=2mm, inner sep=0pt},
			Wnode/.style={circle, draw=black, fill=black, minimum size=1.2mm, inner sep=0pt}]
			\draw[fill=cyan, opacity=0.3] (5,6) circle[radius =0.8cm];
			\draw[fill=cyan, opacity=0.3] (0,4) circle[radius =0.7cm];
			\draw[fill=cyan, opacity=0.3] (3,4) circle[radius =0.7cm];
			\draw[fill=cyan, opacity=0.3] (8,4) circle[radius =0.7cm];
			
			\node [Rnode] (r) at (5,5.8){};
			\node at (5,5.5) {$z$};
			\node at (5,6.2) {$K_{k(d+1)}$};
			
			\draw[line width = 2.5pt] (r)--(0.3,4.7);
			\draw[line width = 2.5pt] (r)--(3.1,4.7);
			\draw[line width = 2.5pt] (r)--(7.8,4.7);
			
			\draw[fill=magenta,opacity=0.5] (-0.33,3.9) circle[radius =0.31cm];
			\draw[fill=green,opacity=0.5] (0,3.6) circle[radius =0.3cm];
			\draw[fill=yellow,opacity=0.5] (0.33,3.9) circle[radius =0.3cm];
			
			\node at (-0.31,4) {\tiny $K_{k-1}$};
			
			\filldraw[magenta, opacity=0.5] (-3.5,2) ellipse (0.2cm and 0.6cm);
			\filldraw[magenta, opacity=0.5] (-3,2) ellipse (0.2cm and 0.6cm);
			\node at (-2.5,2) {$\cdots$};
			\filldraw[magenta, opacity=0.5] (-2,2) ellipse (0.2cm and 0.6cm);
			
			\node at (-3.7,2.2) {\tiny $D_{1,1,1}$};
			\node at (-2.9,2.2) {\tiny $D_{1,1,2}$};
			\node at (-1.8,2.2) {\tiny $D^{1,1,t}$};

			\filldraw[fill=cyan, fill opacity=0.3] (-3.7,0) rectangle (-1.8,1.8);
			\node[Wnode] at (-3.5,1.6){};
			\node[Wnode] at (-3,1.6){};
			\node[Wnode] at (-2,1.6){};
			\node at (-2.7,0.8){$H(t,d,k)$};
			
			\draw[line width = 2.5pt] (-0.6,3.9)--(-3.5,2.6);
			\draw[line width = 2.5pt] (-0.5,3.8)--(-2.9,2.6);
			\draw[line width = 2.5pt] (-0.45,3.7)--(-2,2.6);
			
			\filldraw[green, opacity=0.5] (-1,2) ellipse (0.2cm and 0.6cm);
			\filldraw[green, opacity=0.5] (-0.5,2) ellipse (0.2cm and 0.6cm);
			\node at (0,2) {$\cdots$};
			\filldraw[green, opacity=0.5] (0.5,2) ellipse (0.2cm and 0.6cm);
			\node at (-1,2.2) {\tiny $K_{d+1}$};
			\node at (-0.4,2.2) {\tiny $K_{d+1}$};
			\node at (0.5,2.2) {\tiny $K_{d+1}$};
			
			\draw[fill=cyan, fill opacity=0.3] (-1.2,0) rectangle (0.7,1.8);
			\node[Wnode] at (-1,1.6){};
			\node[Wnode] at (-0.5,1.6){};
			\node[Wnode] at (0.5,1.6){};
			\node at (-0.3,0.8){$H(t,d,k)$};
			
			\draw[line width = 2.5pt] (-0.15,3.5)--(-1,2.65);
			\draw[line width = 2.5pt] (0,3.5)--(-0.5,2.65);
			\draw[line width = 2.5pt] (0.1,3.5)--(0.5,2.65);
			
			\filldraw[yellow, opacity=0.8] (2.5,2) ellipse (0.2cm and 0.6cm);
			\filldraw[yellow, opacity=0.8] (3,2) ellipse (0.2cm and 0.6cm);
			\node at (3.5,2) {$\cdots$};
			\filldraw[yellow, opacity=0.8] (4,2) ellipse (0.2cm and 0.6cm);
			\node at (2.3,2.2) {\tiny $D_{1,q,1}$};
			\node at (3.1,2.2) {\tiny $D_{1,q,2}$};
			\node at (4.2,2.2) {\tiny $D_{1,q,t}$};
			
			\node[Wnode] at (2.5,1.6){};
			\node[Wnode] at (3,1.6){};
			\node[Wnode] at (4,1.6){};
			\draw[fill=cyan, fill opacity=0.3] (2.3,0) rectangle (4.2,1.8);
			\node at (3.3,0.8){$H(t,d,k)$};
			
			\draw[line width = 2.5pt] (0.33,3.7)--(2.5,2.65);
			\draw[line width = 2.5pt] (0.45,3.8)--(3,2.65);
			\draw[line width = 2.5pt] (0.5,4)--(4,2.65);
			
			\filldraw[yellow, opacity=0.8] (8,2) ellipse (0.2cm and 0.6cm);
			\filldraw[yellow, opacity=0.8] (8.5,2) ellipse (0.2cm and 0.6cm);
			\node at (9,2) {$\cdots$};
			\filldraw[yellow, opacity=0.8] (9.5,2) ellipse (0.2cm and 0.6cm);
			\node at (7.8,2.2) {\tiny $D_{k,q,1}$};
			\node at (8.6,2.2) {\tiny $D_{k,q,2}$};
			\node at (9.7,2.2) {\tiny $D_{k,q,t}$};
			
			\draw[fill=yellow, opacity=0.8] (8.32,3.8) circle[radius =0.31cm];
			\node at (8.32,3.85) {\tiny $K_{k-1}$};
			\node[Wnode] at (8,1.6){};
			\node[Wnode] at (8.5,1.6){};
			\node[Wnode] at (9.5,1.6){};
			\draw[fill=cyan, fill opacity=0.3] (7.8,0) rectangle (9.7,1.8);
			\node at (8.8,0.8){$H(t,d,k)$};
			
			\draw[line width = 2.5pt] (8.2,3.6)--(8,2.63);
			\draw[line width = 2.5pt] (8.4,3.6)--(8.5,2.65);
			\draw[line width = 2.5pt] (8.5,3.7)--(9.5,2.65);
			
			\node at (1.5,1.2) {$\cdots$};
			\node at (5.5,3.8) {$\cdots$};
			\node at (6,1.2) {$\cdots$};
			
			\node at (0,4.4) {$C_1$};
			\node at (3,4.4) {$C_2$};
			\node at (3,3.9) {\tiny $K_{(k-1)(d+1)}$};
			\node at (8,4.4) {$C_k$};
			\draw [decorate, very thick, decoration = {calligraphic brace,mirror,amplitude=8pt}] (-3.7,-0.2)--  (4.2,-0.2);
			\node at (0.3,-1) {$q=\binom{(k-1)(d+1)}{k-1} ~ \text{copies}$};
			
			\node at (6.6,6) {$G_1$};
			\draw[gray!60] (8.7,6.6)--(8.9,6.6)--(8.9,3)--(8.7,3);
			\node at (9.2,5) {$G_2$};
			\draw[gray!60] (9.6,6.6)--(9.9,6.6)--(9.9,1.3)--(9.6,1.3);
			\node at (10.2,4) {$G_3$};
		\end{tikzpicture}
		
		\caption{Illustration of the main construction}
		\label{fig-main}
	\end{figure}

	We first show that $G$ is not $(k,d,\ell+1)$-choosable. We shall construct a $k$-list assignment $L$ of $G$, with $L(v) \subseteq [\ell+1]$ for each vertex $v$,  such that $G$ is not $d$-defective $L$-colorable. 
	
	For each vertex $v$ in the subgraph $G_2$, let $L(v)=\{1,2,\ldots,k\}$. For $i=1,2,\ldots,k$, $1 \leq j \leq q$, $1 \leq s \leq t$, let $\theta_{i,j}: [t] \to \{  i, k+1,k+2,\ldots,\ell+1\}$ be a bijection.  For $v \in D_{i,j,s}$, let $L(v) = ([k] \setminus \{i\}) \cup \{\theta_{i,j}(s)\}$.

	For $1 \le i \le k$ and $1 \le j \le q$, for $v \in V(H_{i,j}(t,d,k))-T_{i,j}$, let $L(v)=L_{i,j}(v)$, where $L_{i,j}$  is a list assignment defined as follows: 
	
	Let $\phi_{i,j}$ be the coloring of $T_{i,j}$ defined as 
	$\phi_{i,j}(u_{i,j,s}) = \theta_{i,j}(s)$.  By Lemma \ref{lem-Htdk},   there is a $k$-list assignment $L_{i,j}$ of  $H_{i,j}(t,d,k)$ with  
	$ L_{i,j}(v) \subseteq [\ell+1]$ such that $\phi_{i,j}$ cannot be extended to  a $d$-defective coloring $\psi_{i,j}$ of $H_{i,j}(t,d,k)$ with $\lambda_{H_{i,j}(t,d,k)}(u_{i,j,s},\psi_{i,j})=0$ for each $u_{i,j,s}\in T_{i,j}$.

	Now we prove that $G$ is not $d$-defective $L$-colorable. Assume to the contrary that $G$ has an $L$-coloring $\psi$ with defect $d$. Assume $\pi(z)=r$. Since $G_1$ is a complete graph with $k(d+1)$ vertices, there must be $d$ neighbors of $z$ in $G_1$   colored with $r$. Then the $(k-1)(d+1)$ vertices of $C_r$  are colored with $\{1,2,\ldots, k\}\setminus \{r\}$. Hence  each of the $k-1$ colors is used $d+1$ times in $\psi$. Assume $X_{r,p}$ is a  $(k-1)$-subset   of $C_r$ that are colored by distinct colors from  $\{1,2,\ldots, k\}\setminus \{r\}$. Note that each vertex $x \in X_{r,p}$ has $d$ neighbors in $C_r$ that are colored the same color as $x$. 
	As  all vertices  in $D_{r,p,1}, D_{r,p,2}, \ldots,D_{r,p,t}$ are adjacent  to $X_{r,p}$. For each $1 \leq s \leq t$,  we have $\psi(v) = \theta_{r,p}(s)$ for each vertex  $v \in V(D_{r,p,s})$. 
	In particular, $\psi(u_{r,p,s}) = \theta_{r,p}(s)$ and $u_{r,p,s}$ has $d$ neighbors in $D_{r,p,s}$ that are colored the same color as $u_{r,s,p}$. This implies that $\lambda_{H_{r,p}(t,d,k)}(u_{r,p,s},\psi)=0$ for $s=1,2,\ldots, t$. This is a contradiction, as  $\psi(u_{r,p,s}) = \theta_{r,p}(s) = \phi_{r,p}(u_{r,p,s})$, and by  Lemma \ref{lem-Htdk},  $\phi_{r,p}$ cannot be extended to a $d$-defective $L$-coloring of $H_{r,p}(t,d,k)$ with $\lambda_{H_{r,p}(t,d,k)}(u_{r,p,s},\psi)=0$ for $s=1,2,\ldots, t$.
	
	Lastly, we show that $G$ is $(k,d, \ell)$-choosable. Assume $L$ is any $k$-list assignment of $G$ with $  L(v) \subseteq [\ell]$. We shall show that $G$ is $d$-defective $L$-colorable.
	
	By Lemma~\ref{lem-max-degree},  there is a $d$-defective $L$-coloring $\phi$ of $G_1$.  By Lemma \ref{lem-extend}, $\phi$ can be extended to a $d$-defective $L$-coloring of $G_2$ (by setting $A=V(G_1)$ and $B=V(G_2)\setminus V(G_1)$). Similarly, by Lemma \ref{lem-extend}, we can extend $\phi$   to a $d$-defective $L$-coloring of $G_3$. Note that each vertex in $D_{i,j,s}$ has $(k-1)$ neighbors in $C_i$ and $d$ neighbors in $D_{i,j,s}$, so it also satisfies the inequality in Lemma~\ref{lem-extend} by setting $A=V(G_2)$ and $B=V(G_3)\setminus V(G_2)$. 
	
	It remains to show that we can extend $\phi$ to  $H_{i,j}(t,d,k)$ for $1 \le i \le k, 1 \le j \le q$.

	First assume that for any $u,v \in X_{i,j}$, $\phi(u) \neq \phi(v)$ and for any $x \in X_{i,j}$, $\lambda_{G_3}(x,\phi)=d$. In this case, there are only $\ell-(k-1)=t-1$ colors available for vertices in $\bigcup^t_{s=1}V(D_{i,j,s})$. So at least two of vertices in $T_{i,j}$ received the same color By Lemma~\ref{lem-Htdk}, we can extend $\phi$ to $H_{i,j}(t,d,k)$. In particular, if 
	$d=0$, then $\phi$ is a $0$-defective coloring, i.e., a proper coloring of $G_3$, implying that for any $u,v \in X_{i,j}$, $\phi(u) \neq \phi(v)$ and for any $x \in X_{i,j}$, $\lambda_{G_3}(x,\phi)=0=d$. Hence $\phi$ can be extended to $H_{i,j}(t,d,k)$.
	
	Next assume that $\phi(u) = \phi(v)$ for some $u,v \in X_{i,j}$ or  $\lambda_{G_3}(x,\phi) \leq d-1$ for some $x \in X_{i,j}$. In either case, it is enough  to modify $\phi$ on $D_{i,j,t}$ (if necessary) such that $\lambda_{G_3}(u_{i,j,t},\phi) \leq d-1 $. So it suffices to show that we can extend $\phi$ from $T$ to $H_{i,j}(t,d,k)$ such that $\lambda_{H_{i,j}(t,d,k)}(u_{i,j,s})=0$ for $s=1,2,\ldots,t-1$ and $\lambda_{H_{i,j}(t,d,k)}(u_{i,j,t})\leq 1$. By Lemma~\ref{lem-Htdk}, this is possible.  The proof of Theorem~\ref{main-th2} now is completed. \qed
	
	\medskip
	
	\section{Open problems}
	
	For $d \geq 0$, a graph  is $d$-defective $1$-choosable if and only if $\Delta(G) \leq d$, and which is equivalent to be  $d$-defective $1$-colorable. So  Question \ref{Q-K} has a positive answer for $k=1$ and $\ell_{1,d}=1$.
	For $k=2$ and $d=0$,   Kr\'{a}l' and Sgall \cite{KS2005} showed that every $(2,0,4)$-choosable graph is $(2,0,+\infty)$-choosable. The proof relies on a characterization of $2$-choosable graphs by \cite{ERT1980}. For $d \ge 1$, no characterization of $d$-defective $2$-choosable graphs is known, and Question \ref{Q-K} remains open for $k=2$ and $d \ge 1$.

	The following problem asked by Kr\'{a}l and Sgall \cite{KS2005}  also remains open.
	
	\begin{problem}
		Is it true that for each $k \geq 3$, there exists a number $\ell$ such that each $(k,0,\ell)$-choosable graph is $(k+1)$-choosable? If so, what is the least number $\ell$ with this property?
	\end{problem}

	
	\bibliographystyle{abbrv}
	\bibliography{reference}
	
\end{document}